\documentclass[11pt,reqno]{amsart}

\usepackage{amssymb, amsmath, amsthm,amscd,mathabx,amsxtra}
\usepackage[colorlinks,plainpages,backref=page,urlcolor=blue]{hyperref}
\hypersetup{breaklinks=true,  colorlinks=true,  pdfusetitle=true}
\usepackage[alphabetic,lite]{amsrefs}
\usepackage{mathtools}
\usepackage{amscd}
\usepackage{tikz}
\usetikzlibrary{cd,calc,arrows}
\usepackage{mathrsfs}
\usepackage{cases}
\usepackage[shortlabels]{enumitem}
\usepackage{breakurl}
\usepackage{xurl}

\newcommand{\G}{\Gamma}

\newcommand{\Wb}{W}

\newcommand{\N}{\mathbb{N}}

\renewcommand{\k}{\Bbbk}

\newcommand{\RR}{\mathcal{R}}

\newcommand{\Rt}{\RR}
\newcommand{\SR}{\k\langle \Delta \rangle}

\newcommand{\Ann}{\operatorname{Ann}}

\newcommand{\Spec}{\operatorname{Spec}}
\newcommand{\Sym}{\operatorname{Sym}}
\newcommand{\Tor}{\operatorname{Tor}}

\newcommand{\Hilb}{\operatorname{Hilb}}

\newcommand{\Fitt}{\operatorname{Fitt}}
\newcommand{\Grass}{\operatorname{Gr}}

\newcommand{\Supp}{\operatorname{Supp}}
\newcommand{\pdim}{\operatorname{pdim}}
\newcommand{\reg}{\operatorname{reg}}
\newcommand{\sqf}{\operatorname{sqf}}

\newcommand{\coker}{\operatorname{coker}}

\newcommand{\rank}{\operatorname{rank}}
\newcommand{\im}{\operatorname{im}}
\newcommand{\id}{\operatorname{id}}

\newcommand{\lk}{\operatorname{lk}}

\newcommand{\abs}[1]{\lvert#1\rvert}

\newcommand{\bwedge}{\mbox{\normalsize $\bigwedge$}}

\definecolor{dkgreen}{RGB}{0,100,0}
\definecolor{dkbrown}{RGB}{139,69,19}
\DeclareMathAlphabet\mathbfcal{OMS}{cmsy}{b}{n}

\newtheorem{theorem}{Theorem}[section]
\newtheorem*{theorem*}{Theorem}
\newtheorem*{problem*}{Problem}

\newtheorem{proposition}[theorem]{Proposition}
\newtheorem{corollary}[theorem]{Corollary}
\newtheorem*{corollary*}{Corollary}

\theoremstyle{definition}
\newtheorem{definition}[theorem]{Definition}
\newtheorem*{definition*}{Definition}
\newtheorem{example}[theorem]{Example}

\newtheorem{remark}[theorem]{Remark}
\newtheorem*{remark*}{Remark}
\newtheorem*{notation*}{Notation}
\newtheorem*{ack}{Acknowledgments}

\numberwithin{equation}{section}

\begin{document}

\title[Resonance and Koszul modules of simplicial complexes]%
{Higher resonance schemes and Koszul modules of simplicial complexes}

\author[M. Aprodu]{Marian Aprodu}
\address{Marian Aprodu: Simion Stoilow Institute of Mathematics \hfill \newline
\indent P.O. Box 1-764,
RO-014700 Bucharest, Romania, and \hfill  \newline
\indent  Faculty of Mathematics and Computer Science, University of Bucharest,  Romania}
\email{{\tt marian.aprodu@imar.ro}}

\author[G. Farkas]{Gavril Farkas}
\address{Gavril Farkas: Institut f\"ur Mathematik, Humboldt-Universit\"at zu Berlin \hfill \newline
\indent Unter den Linden 6,
10099 Berlin, Germany}
\email{{\tt farkas@math.hu-berlin.de}}

\author[C. Raicu]{Claudiu Raicu}
\address{Claudiu Raicu: Department of Mathematics,
University of Notre Dame \hfill \newline
\indent 255 Hurley Notre Dame, IN 46556, USA, and \hfill\newline
\indent Simion Stoilow Institute of Mathematics, \hfill\newline
\indent  P.O. Box 1-764, RO-014700 Bucharest, Romania}
\email{{\tt craicu@nd.edu}}

\author[A. Sammartano]{Alessio Sammartano}
\address{Alessio Sammartano: Dipartimento di Matematica, Politecnico di Milano \hfill\newline
\indent Via Bonardi 9, Milan, 20133, Italy}
\email{{\tt alessio.sammartano@polimi.it}}

\author[A. Suciu]{Alexander I. Suciu}
\address{Alexander I. Suciu: Department of Mathematics,
Northeastern University \hfill \newline
\indent Boston, MA, 02115, USA}
\email{{\tt a.suciu@northeastern.edu}}

\subjclass[2020]{Primary
13F55. 
Secondary
14M12,  
16E05.  
}

\keywords{Simplicial complex, square-free monomial ideal,
Koszul module, resonance variety, reduced scheme, Hilbert series.}

\begin{abstract}
Each connected graded, graded-commutative algebra $A$ of finite type
over a field $\k$ of characteristic zero defines a complex of finitely generated,
graded modules over a symmetric algebra, whose homology graded modules
are called the \emph{(higher) Koszul modules}\/ of $A$. In this
note, we investigate the geometry of the support loci of these modules,
called the \emph{resonance schemes}\/ of the algebra.
When $A=\k\langle \Delta \rangle$ is the exterior Stanley--Reisner algebra
associated to a finite simplicial complex $\Delta$, we show that
the resonance schemes are reduced. We also compute the
Hilbert series of the Koszul modules and give bounds on
the regularity and projective dimension of these graded modules.
This leads to a relationship between resonance and Hilbert series
that generalizes a known formula for the Chen ranks of a
right-angled Artin group.
\end{abstract}

\maketitle

\section{Introduction and statement of results}
\label{sect:intro}

Koszul modules are graded modules over a symmetric algebra that are constructed from
the classical Koszul complex. They emerged from geometric group theory and
topology \cites{AFPRW, PS-crelle} and found applications in other fields such as
algebraic geometry. One prominent instance is \cite{AFPRW2}, where the effective
vanishing in high degrees of some Koszul modules led to a new proof of the celebrated
Green's Conjecture on syzygies of generic canonical curves. The argument relies on a
connection between the graded pieces of those particular Koszul modules and the Koszul
cohomology of the tangent developable surface of a rational normal curve. The non-trivial
vehicle that permits the passage in \cite{AFPRW2} from symmetric powers (Koszul modules)
to exterior powers (Koszul cohomology) is an explicit version of the Hermite reciprocity formula.

It is the aim of this paper to describe a completely new instance where the passage from
Koszul modules to Koszul cohomology of some homogeneous coordinate ring is still possible.
The setup is however simpler and more elementary than the one involved with Green's Conjecture.

For a ground field $\k$ of characteristic $0$, a classical construction of Stanley and Reisner
associates to every simplicial complex $\Delta$ on $n$ vertices a graded,
graded-commutative algebra $\SR = E/J_{\Delta}$,
where $E=\bwedge_{\k}(e_1,\dots, e_n)$ is the
exterior algebra over $\k$ and $J_{\Delta}$ is the ideal generated by all the 
monomials $e_{\sigma}=e_{j_1}\wedge \cdots \wedge e_{j_s}$ corresponding 
to simplices $\sigma=(j_1,\dots, j_s)$ with $1\le j_1<\cdots <j_s \le n$ 
which do not belong to $\Delta$.

Let $S\coloneqq \k[x_1,\dots, x_n]$ be the polynomial ring in $n$ variables
over $\k$, and consider the cochain complex $(\SR^{\bullet}\otimes_{\k} S, \delta)$
of free, finitely generated, graded $S$-modules
obtained by applying the BGG correspondence to the finitely generated,
graded $E$-module $\SR^{\bullet}$.
The Fitting ideals of this complex define the {\em jump resonance loci}\/
of our simplicial complex,
\begin{equation}
\label{eq:res-intro}
\RR^i(\Delta)\coloneqq V\bigl(\Fitt_{\beta_{i+1}} (\delta^{i-1}\oplus \delta^{i})\bigr),
\end{equation}
where $\beta_{i+1}$ is the number of faces of dimension $i$ in $\Delta$.
It was shown in \cite{PS-adv09} that the irreducible components of
$\RR^i(\Delta)$ are coordinate subspaces of $\SR^1=\k^n$,
given explicitly in terms of the (simplicial) homology groups of certain
subcomplexes of $\Delta$.

Now let $\bigl(\SR_{\bullet} \otimes_{\k} S, \partial\bigr)$ be the dual chain complex,
and define the {\em Koszul modules}\/ (in weight $i$) of the simplicial complex
$\Delta$ to be the homology $S$-modules of this complex,
\begin{equation}
\label{eq:km-intro}
\Wb_i(\Delta)\coloneqq H_i\bigl(\SR_{\bullet} \otimes_{\k} S, \partial\bigr).
\end{equation}
An alternate definition of resonance is given by the support loci
of these modules,
\begin{equation}
\label{eq:res-tilde-intro}
\Rt_i(\Delta) \coloneqq V\bigl(\Ann (W_i(\Delta))\bigr).
\end{equation}
These varieties, called the \emph{support resonance loci}, are again
finite unions of coordinate subspaces.
Though they do not coincide in general with the previously defined
sets $\RR^i(\Delta)$, it is known that $\Rt_1(\Delta)=\RR^1(\Delta)$ 
(away from $0$) and $\bigcup_{j\le i}\Rt_j(\Delta)=\bigcup_{j\le i} \RR^j(\Delta)$
for all $i\ge 1$.

A notable property of the higher Koszul modules associated to simplicial
complexes is that they are multigraded as opposed to the general case
when they are only graded modules. Using the general theory of multi-graded
square-free modules, we prove that the multi-graded pieces of the Koszul modules
can be described as multi-graded pieces of some $\Tor$'s over symmetric algebras.
It is known (see for example \cite{PS-crelle}) that the graded pieces of weight-one
Koszul modules are graded pieces of $\Tor$'s over exterior algebras; however, their
relations with $\Tor$'s over symmetric algebras is quite rare in general.

\begin{theorem}
\label{thm:koszul-delta-intro}
For any $i\ge 1$ and any square-free multi-index $\mathbf{b}$, there is
a natural isomorphism of vector spaces,
\begin{equation}
\label{eqn:Mod-Tor-intro}
\left[\Wb_i(\Delta)\right]_\mathbf{b}
\cong \left[\Tor^S_{|\mathbf{b}|-i}
(\k, \k[\Delta])\right]_\mathbf{b}^\vee ,
\end{equation}	
where $\k[\Delta]$ is the polynomial Stanley--Reisner ring of $\Delta$.
\end{theorem}

We refer to Section \ref{subsec:sq-free} for a quick review of multi-graded
square-free modules. This multigraded structure of the Koszul modules is captured
in the Hilbert series.

\begin{theorem}
\label{thm:hilb-wd-intro}
For every simplicial complex $\Delta$, the multigraded Hilbert series
of the Koszul modules $\Wb_i(\Delta)$ are given by
\begin{equation*}
\label{eq:hilb-intro}
 \sum_{\mathbf{a}\in \N^n} \dim_{\k} [\Wb_i(\Delta)]_\mathbf{a} \,
\mathbf{t}^\mathbf{a}=
 \sum_{\substack{\mathbf{b}\,\in\, \N^n\\ \mathbf{b} \operatorname{square-free}}}
\dim_{\k}(\widetilde{H}_{i-1}(\Delta_\mathbf{b}; \k))
\frac{\mathbf{t}^\mathbf{b}}{\prod_{j \in \Supp(\mathbf{b})}(1-t_j)}.
\end{equation*}
\end{theorem}

In Section \ref{sec:simp-comp}, we give a precise description of the irreducible components
of the support resonance loci. In each weight $i$, they correspond to maximal subcomplexes
with non-vanishing reduced homology in degree $i-1$.
\begin{theorem}
\label{thm:irred-comps-support-intro}
For every simplicial complex $\Delta$ and every $i\ge 1$, the scheme structure on
the support resonance $\Rt_i(\Delta)$ is reduced.
Moreover, the decomposition in irreducible components is given by
	\begin{equation}
		\label{eq:res-delta2-intro}
		\Rt_i(\Delta)=\bigcup_{\substack{
				\mathsf{V'}\subseteq \mathsf{V}\: \mathrm{ maximal\: with }\\[1pt]
				\widetilde{H}_{i-1}(\Delta_{\mathsf{V'}};\k)\neq 0
		}}\k ^{\mathsf{V'}}.
	\end{equation}
\end{theorem}

Particularly interesting is the case when $\Delta$ is $1$-dimensional,
that is, it may be viewed as a finite simple graph $\Gamma$.
It was shown in \cite{PS-mathann} that all the irreducible components
of $\RR^1(\Gamma)$ are coordinate subspaces, which correspond to the
maximally disconnected full subgraphs of $\Gamma$. This result comes as a direct
consequence of our analysis. The statement concerning the reducedness of $\Rt_i(\Delta)$
can be compared with the detailed study performed in \cite{AFRS} on the scheme structure
of (support) resonance varieties associated to classical Koszul modules.

\begin{ack}
{\small{Aprodu was supported by the PNRR grant CF 44/14.11.2022
\emph{Cohomological Hall algebras of smooth surfaces and applications}.
Farkas supported by the DFG Grant \emph{Syzygien und Moduli} and by the
ERC Advanced Grant SYZYGY. This project has received funding from the
European Research Council (ERC) under the EU Horizon 2020  program
(grant agreement No. 834172). Raicu was supported by the NSF Grant No.~2302341.
Sammartano was supported by the grant PRIN 2020355B8Y
\emph{Square-free Gr\"obner degenerations, special varieties and related topics}.
Suciu was supported by Simons Foundation Collaboration Grant for Mathematicians
No.~693825.}}
\end{ack}

\section{Graded algebras, Koszul modules, and higher resonance}
\label{sect:kozul-res}

We start in a more general context (adapted from the setup in
\cites{PS-mrl, Su-indam}), that will be used throughout the paper.

\subsection{Chain complexes associated to graded algebras}
\label{subsec:cc-ga}

Let $A^{\bullet}$ be a graded, graded-com\-mutative algebra over a field
$\k$ of characteristic $0$, with multiplication maps $A^i \otimes_{\k} A^j \to A^{i+j}$.
We will assume that $A$ is connected (that is, $A^0=\k$) and of finite-type (that is,
$\dim_{\k}A^i<\infty$, for all $i>0$), and we will write $\beta_i(A)=\dim_{\k}A^i$.
To avoid trivialities, we always assume that $\beta_1(A)\neq 0$.

For each $a\in A^1$, graded commutativity of multiplication yields
$a^2=0$, therefore, we have a cochain complex
\begin{equation}
\label{eq:cc}
\begin{tikzcd}[column sep=24pt]
(A^{\bullet} , \delta_{a})\colon  \
A^0 \ar[r, "\delta^0_{a}"] & A^1
\ar[r, "\delta^1_{a}"]
& A^2 \ar[r, "\delta^2_{a}"] & \cdots ,
\end{tikzcd}
\end{equation}
with differentials $\delta^i_{a} (u)= a \cdot u$, for all $u \in A^i$.
The {\em resonance varieties}\/ of $A$ are the jump loci for the
cohomology groups of this complex: for each $i\ge 0$, we put
\begin{equation}
\label{eq:res-a}
\RR^i(A)\coloneqq \bigl\{a \in A^1
\mid H^i(A^{\bullet}, \delta_{a}) \ne 0 \bigr\}.
\end{equation}
Clearly, these are homogeneous subsets of the affine space $A^1$.
Since $A^0$ is $1$-dimensional, generated by $1\in \k$, and since
$\delta_a(1)=a$ for each $a\in A^1$, it follows that $\RR^0(A)=\{0\}$.
The most studied is the first resonance variety, which can be described
as the set
\begin{equation}
\label{eq:res1}
\RR^1(A) =\{ a \in A^1 \mid \text{$\exists\, b \in A^1, \
0\neq a\wedge b \in K^{\perp} $}\} \cup \{0\},
\end{equation}
where $K^{\perp}$ denotes the kernel of the multiplication map $
A^1\wedge A^1 \to A^2$.

Let us now fix a $\k$-basis $\{ e_1,\dots, e_n \}$ of $A^1$ and let
$\{ x_1,\dots, x_n \}$ be the dual basis of the dual $\k$-vector
space $A_1=(A^1)^{\vee}$. This allows us to identify the symmetric
algebra $\Sym(A_1)$ with the polynomial ring $S=\k[x_1,\dots, x_n]$,
the coordinate ring of the affine space $A^1\cong \k^{\beta_1(A)}$.

Viewing $A^{\bullet}$ as a graded module over the exterior algebra
$E^{\bullet}=\bwedge A^1$, the BGG correspondence \cite{Ei-syz}
yields a cochain complex of finitely generated, free $S$-modules,
\begin{equation}
\label{eq:cc-s}
\begin{tikzcd}[column sep=20pt]
\hspace*{-12pt} \bigl(A^{\bullet} \otimes_{\k} S,\delta_A\bigr)\colon
 \cdots \ar[r]
&A^{i}\otimes_{\k} S \ar[r, "\delta^{i}_A"]
&A^{i+1} \otimes_{\k} S \ar[r, "\delta^{i+1}_A"]
&A^{i+2} \otimes_{\k} S \ar[r]
& \cdots,
\end{tikzcd}
\end{equation}
whose coboundary maps are the $S$-linear maps given by
$\delta^{i}_A(u \otimes s)= \sum_{j=1}^{n} e_j u \otimes s x_j$
for $u\in A^i$ and $s\in S$.
It is readily seen that this cochain complex is independent of the
choice of basis for $A^1$ and that, moreover,
the specialization of $(A\otimes_{\k} S,\delta_A)$
at an element $a\in A^1$ coincides  with the complex
$(A,\delta_a)$ defined by \eqref{eq:cc}. 

\begin{remark}
\label{rem:grading}
Typically, the elements in the exterior algebra $E^{\bullet}$ 
are given negative degrees, see for instance \cite{Ei-syz}. 
However, we prefer to work here with the positive grading, 
which amounts to negating the standard action of the torus $(\k^*)^n$ 
on $E^{\bullet}$ as well.  This convention carries over to the exterior 
Stanley--Reisner rings $A^{\bullet}=\SR$ in Section \ref{subsec:SR-ring} 
and is compatible with the notation of \cite{AAH}, as recalled in 
Proposition~\ref{prop:aah-hochster}.
\end{remark}

It follows directly from the definition \eqref{eq:res-a} that a point
$a \in A^1$ belongs to $\RR^i(A)$ if and only if
$\rank \delta^{i-1}_a + \rank \delta^{i}_a < \beta_i(A)$.
Therefore,
\begin{equation}
\label{eq:res-fitt}
\RR^i(A)= V \big( \Fitt_{\beta_{i+1}(A)} \big(\delta^{i-1}_A\oplus \delta^{i}_A\big) \big),
\end{equation}
where $\psi_1\oplus \psi_2$ denotes the block sum of two matrices,
$\Fitt_r(\psi)$ denotes the ideal of minors of size $n-r$ of a matrix 
$\psi\colon S^m\to S^n$, and $V(I)$ denotes the zero-set of an ideal 
$I\subset S$. This shows that the sets $\RR^i(A)$ are algebraic 
subvarieties of the affine space $A^1$ called \emph{jump resonance loci}.

\subsection{Koszul modules and their support loci}
\label{subsec:kozul-res}

Set $A_i\coloneqq (A^i)^{\vee}$ and $\partial^A_i\coloneqq (\delta_A^{i-1})^{\vee}$
and consider the chain complex of finitely generated $S$-modules
\begin{equation}
\label{eq:a-tensor-s}
\begin{tikzcd}[column sep=24pt]
\bigl(A_{\bullet} \otimes_{\k} S,\partial\bigr)\colon
 \cdots \ar[r]
&A_{i+1}\otimes_{\k} S \ar[r, "\partial_{i+1}^A"]
&A_{i} \otimes_{\k} S \ar[r, "\partial_{i}^A"]
&A_{i-1} \otimes_{\k} S \ar[r]
& \cdots.
\end{tikzcd}
\end{equation}
We define the {\em Koszul modules (in weight $i$)}\/ of the algebra $A$ as
the homology $S$-modules of this chain complex, that is,
\begin{equation}
\label{eq:wi-a}
\Wb_i(A) \coloneqq H_i\bigl(A_{\bullet}\otimes_{\k} S\bigr) .
\end{equation}

Clearly, these are finitely generated, graded $S$-modules. The degree $d$
component of the Koszul module $\Wb_i(A)$ is computed by the homology of the complex
\begin{equation}
\label{eq:les-as}
\begin{tikzcd}[column sep=18pt]
A_{i+1}\otimes_\k S_{d-i-1} \ar[r]& A_i\otimes_\k S_{d-i} \ar[r]& A_{i-1}\otimes_\k S_{d-i+1},
\end{tikzcd}
\end{equation}
where we recall that $S=\Sym(A^1)$.  It follows straight from the definitions
that $\Wb_0(A)=\k$ is the trivial $S$-module.

Setting $E_{\bullet}\coloneqq \bwedge A_1$, the first Koszul module also has
the following presentation
\begin{equation}
\label{eq:pres-w1a}
\begin{tikzcd}[column sep=18pt]
\big(E_3 \oplus K \big)  \otimes_{\k}  S
\ar[r, "\partial_3^E +\iota  \otimes_{\k} \id_S"] &[32pt]
E_2  \otimes_{\k}  S \ar[r, two heads] & \Wb_1(A),
\end{tikzcd}
\end{equation}
where
$\begin{tikzcd}[column sep=18pt]
\hspace*{-5pt} K=\bigl\{\varphi \in A_1\wedge A_1=(A^1\wedge A^1)^{\vee} \mid
\varphi_{|K^{\perp}}\equiv 0\bigr\}  \ar[r, hook, "\iota"] & A_1\wedge A_1 =E_2.
\end{tikzcd}$

The {\em resonance schemes}\/ of the graded algebra $A$ are defined by
the annihilator ideals of the Koszul modules of $A$,
\begin{equation}
\label{eq:ria-spec}
\Rt_i(A) \coloneqq \Spec \bigl(S/ \Ann \Wb_i(A)\bigr).
\end{equation}

By slightly abusing notation, we also  denote by $\Rt_i(A)=\Supp \Wb_i(A)$ the
underlying sets and call them \emph{support resonance loci}.
Note that the algebra structure on $A^\bullet$ is not essential in the discussion
above, as the definitions of Koszul modules and support resonance loci only
use the $E$-module structure. In particular, the constructions apply for
finitely-generated graded $E$-modules, as well.

Clearly $\Rt_0(A)=\RR^0(A)=\{0\}$. More generally,
suppose $\Wb_j(A)\neq 0$ for all $1\le j\le i$.
Then, as shown in \cite[Theorem 2.5]{PS-mrl}, the support resonance loci
are related to the jump resonance loci by the formula%
\footnote{We denote by $\Rt^i(A)$ what in \cite{PS-mrl} is denoted
by $\Rt^i_1(A)=\mathcal{V}^i_1(A^{\bullet} \otimes_{\k} S )$ and in
\cite{Su-indam} by $\Rt^i(A)$, whereas
we use the notation $\Rt_i(A)$ for what in \cite{PS-mrl} is denoted by
$\mathcal{W}^i_1(A)=\Supp H_i(A_{\bullet} \otimes_{\k} S)$
and in \cite{Su-indam} by $\widetilde{\RR}_i(A)$.}
\begin{equation}
\label{eq:union-res}
\bigcup_{j\le i} \Rt_j(A) = \bigcup_{j\le i} \RR^j(A).
\end{equation}
In particular, if $\Wb_1(A)\neq 0$, then $\Rt_1(A) =\RR^1(A)$.

\subsection{Quotients of exterior algebras through ideals generated in fixed degree}
\label{subsec:fixed-deg}

We now discuss a particularly interesting case of this general construction.
Fix integers $d\ge 1$ and $n\ge 3$.
Let $V$ be an $n$-dimensional vector space over the field $\k$ and let
$K\subseteq \bwedge^{d+1}V$ be a subspace. Set $S\coloneqq \Sym(V)$
and $E^\bullet\coloneqq \bwedge V^\vee$, and
then consider the linear subspace
\begin{equation}
\label{eq:kperp}
K^\perp\coloneqq \big( \bwedge^{d+1}V/K \big)^\vee =
\big\{\varphi\in \bwedge^{d+1}V^\vee \mid \varphi_{|K}= 0\big\}
\subseteq \bwedge^{d+1}V^\vee.
\end{equation}
Letting
$A^\bullet\coloneqq E^\bullet/\langle K^\perp\rangle$ be the quotient of the exterior
algebra $E^\bullet$ by the (homogeneous) ideal generated by $K^\perp$,
we clearly have $K=A_{d+1}$. Conversely, if $J\subseteq E^\bullet$
is a homogeneous ideal generated in degree $d+1$ and we take
$K^\perp\coloneqq J_{d+1}$, then the algebra $A^\bullet=E^\bullet/J$ is
obtained as above. Denote by $j$ the inclusion of the dual algebra
$A_\bullet$ into $E_\bullet$. Recalling that
$\partial_i \colon \bwedge^i V\otimes_{\k} S\rightarrow \bwedge^{i-1} V\otimes_{\k} S$
is the Koszul differential, we have the following characterization.

\begin{proposition}
\label{prop:Koszul-for-fixed-degree}
The Koszul modules $\Wb_i(V,K)=\Wb_i(A)$ satisfy the following properties:
\begin{enumerate}[itemsep=2pt, topsep=-1pt]
\item \label{wb1}
$\Wb_i(A)=0$ for $i\le d-1$.
\item \label{wb2}
$\Wb_d(A)=\coker \big(\partial_{d+2}+j_{d+1}  \otimes_{\k}  S\big)$.
\end{enumerate}
\end{proposition}

\begin{proof}
The first part is quite straightforward, as $J_i=0$ for $i\le d-1$ and hence
$A_i=E_i$ for $i\le d-1$. For the second part, first note that the $d$-th
Koszul module is in this case the middle homology of the complex
\begin{equation}
\label{eq:dK-module}
\begin{tikzcd}[column sep=18pt]
K\otimes_{\k} S\ar[r] & \bwedge^dV\otimes_{\k} S \ar[r]& \bwedge^{d-1}V\otimes_{\k} S,
\end{tikzcd}
\end{equation}
and hence
\begin{equation}
\label{eq:WdA}
\begin{tikzcd}[column sep=18pt]
\Wb_d(V,K)=\coker\Bigl\{K\otimes_{\k} S \ar[r, "j_{d+1}  \otimes_{\k} S"]
&[24pt] \bwedge^{d+1}V\otimes_{\k} S \ar[r, two heads]& \coker(\partial_{d+2})\Bigr\}.
\end{tikzcd}
\end{equation}
Applying now the Snake Lemma to the diagram
\begin{equation}
\label{eq:snake-1}
\begin{tikzcd}[column sep=16pt]
0 \ar[r] & \bwedge^{d+2}V\otimes_{\k} S \ar[r] \ar[d, two heads]
& \bigl(\bwedge^{d+2}V \oplus K\bigr)\otimes_{\k} S\ar[r]
\ar[d, "\partial_{d+2} + j_{d+1}\otimes_{\k} S"]
& K\otimes_{\k} S\ar[r]\ar[d] & 0 \\
0 \ar[r]& \im(\partial_{d+2}) \ar[r]& \bwedge^{d+1}V\otimes_{\k} S \ar[r]
&\coker(\partial_{d+2})\ar[r]&0
\end{tikzcd}
\end{equation}
establishes the claim.
\end{proof}

\begin{remark}
\label{rem:snake}
Note that the Snake Lemma also applies to the diagram
\begin{equation}
\label{eq:snake-2}
\begin{tikzcd}[column sep=16pt]
	0 \ar[r] & K\otimes_{\k} S\ar[r]\ar[d, equal] 
	& \bigl(\bwedge^{d+2}V \oplus K\bigr) \otimes_{\k} S\ar[r]\ar[d]
	& \bwedge^{d+2}V\otimes_{\k} S \ar[r]\ar[d] & 0 \\
	0 \ar[r]& K\otimes_{\k} S \ar[r]& \bwedge^{d+1}V\otimes_{\k} S \ar[r]
	&(\bwedge^{d+1}V/K)\otimes_{\k} S\ar[r]&0,
\end{tikzcd}
\end{equation}
leading to the simpler presentation
\begin{equation}
	\label{eqn:simple-presentation}
	\begin{tikzcd}[column sep=18pt]
		\Wb_d(A)=\coker\Bigl\{\bwedge^{d+2}V\otimes_{\k} S\ar[r]
		& (\bwedge^{d+1}V/K)\otimes_{\k} S\Bigr\}.
	\end{tikzcd}
\end{equation}
\end{remark}

If $d=1$, in weight $1$ we recover the original Koszul module $W(V,K)$ of a pair $(V,K)$ with
$K\subseteq \bwedge^2V$ considered in \cites{AFPRW, AFRS, PS-crelle} and elsewhere.
However, note the shift by two in degrees, that is, $W(V,K)=\Wb_1(A)(2)$.

\begin{example}
\label{ex:vb}
Let $X$ be a smooth complex projective variety, and consider a vector bundle
$E$ on $X$ of rank $\ge r+1$, for some integer $r\ge 1$. We consider the
determinant maps
	\begin{equation}
		\label{eqn:drmap}
		\begin{tikzcd}[column sep=18pt]
			d_r\colon \bwedge^{r+1}H^0(X,E)\ar[r]& H^0(X,\bwedge^{r+1}E)
		\end{tikzcd}
	\end{equation}
and take $K_r^\perp\coloneqq \ker(d_r)$. Then the above construction applies,
producing for each $r$ a series of Koszul modules
$\Wb_r(X,E)\coloneqq \Wb_r\bigl(H^0(X, E)^{\vee}, K_r \bigr)$.
\end{example}

As in the case $d=1$ (see \cites{AFPRW2, PS-crelle}), we have a geometric
characterization of vanishing resonance, in which case the corresponding
Koszul module is of finite length and hence vanishes in high degrees.
For an element $\omega\in \bwedge^{d+1}V^\vee$, we denote by
$\varphi(\omega)\colon V^\vee\to\bwedge^{d+2}V^\vee$ the map
$a\mapsto a\wedge\omega$. Consider the projective variety parameterizing
the decomposable elements,
\begin{equation}
\label{eq:Sigma}
\Sigma_d\coloneqq \bigl\{[\omega]\in\mathbb{P} \big(\bwedge^{d+1}V^\vee \big) \mid
\rank (\varphi(\omega))\le n-1\bigr\}.
\end{equation}

Standard multilinear algebra proves the following proposition.
\begin{proposition}
\label{prop:rtd-a}
If $d\ge 2$, then $\Rt_d(A)=\{0\}$ if and only if
$\mathbb{P}(K^\perp)\cap \Sigma_d=\emptyset$.
\end{proposition}

 \proof
Recall that $a\in V^\vee$ divides $\omega$ if and only if $a\wedge\omega=0$.
Therefore, $\Rt_d(A)\neq\{0\}$ if and only if there exist $a\in V^\vee$ and
$b\in \bwedge^dV^\vee$ such that $0\neq a\wedge b\in K^\perp$. This is
equivalent to the existence of a non-zero element $a\in V^\vee$ and of a
non-zero element $\omega\in K^\perp$ such that $a\wedge\omega=0$, i.e.,
$0\neq a\in \ker(\varphi(\omega))$, and hence $[\omega]\in \Sigma_d$.
 \endproof

The case $d=1$ is special, since $\varphi(\omega)$ non-injective implies its kernel is
at least $2$-dimensional. Indeed, if $\omega =a\wedge b\neq 0$ then $\ker(\varphi(\omega))$
is generated by $a$ and $b$. In this case, $\Sigma_1$ is the Grassmann variety
$\Grass_2(V^\vee)\subseteq \mathbb{P}\bigl(\bwedge^2 V^{\vee}\bigr)$.

\begin{remark}
\label{rem:detmap}
For the Koszul module $\Wb_r(X,E)$ considered in Example \ref{ex:vb}, we have that
the resonance $\Rt_r(X,E)\coloneqq \Supp \Wb_r(X,E)$ is non trivial if and only if there
exists a section $0\neq s\in H^0(X,E)$ such that the determinant map
$d_s\colon \bwedge^r H^0(X,E)\rightarrow \bwedge^{r+1} H^0(X,E)$
given by $\omega\mapsto d_{r+1}(s\wedge \omega)$ is not injective.
\end{remark}

\section{Simplicial complexes and their Koszul modules}
\label{sect:res-sc}

\subsection{Square-free modules}
\label{subsec:sq-free}
We start this section with some algebraic preliminaries regarding
square-free modules. We recall from \cite{Ya00} some basic facts
about this type of modules, which will be needed in
Sections  \ref{subsec:hilb}, \ref{sec:simp-comp}, and \ref{subsec:reg}.

Let $V$ be a $\k$-vector space of dimension $n$, and identify the symmetric
algebra $\Sym(V)$ with the polynomial ring $S=\k[x_1,\ldots,x_n]$.
We consider the standard $\N^n$-multigrading on $S$,
defined by $\deg(x_i) = \mathbf{e}_i \in \N^n$,
where $\mathbf{e}_i=(0,\ldots,1,\ldots,0)$ is the multi-index
with $1$ placed in the $i$-th position.
Given a multi-index $\mathbf{a}=(a_1, \ldots, a_n) \in \N$,
its support is defined as the set
$\Supp(\mathbf{a})\coloneqq \{ i \mid a_i > 0\}$.

\begin{definition}
\label{defn:Nn-square-free}
An $\N^n$-graded $S$-module $M$ is said to be \emph{square-free}\/
if for any $\mathbf{a}\in\N^n$ and any $i\in\Supp(\mathbf{a})$,
the multiplication map
\begin{equation}
\label{eq:mult-map}
\begin{tikzcd}[column sep=16pt]
x_i \colon M_{\mathbf{a}} \ar[r]& M_{\mathbf{a}+\mathbf{e}_i}
\end{tikzcd}
\end{equation}
is an isomorphism.
\end{definition}

This definition is a direct generalization of the case of ideals.
Indeed, an ideal $I \subseteq S$ is a square-free module if and only if it is a
square-free monomial ideal, and this is also equivalent to $S/I$
being a square-free module.

Note that a free $\N^n$-graded $S$-module is square-free if
and only it is  generated in square-free multidegrees.

\begin{proposition}
\label{prop:Nn-sqf-ker-coker}
If $f\colon M\to N$ is a morphism of $\N^n$-graded $S$-modules,
and $M$ and $N$ are square-free modules, then $\ker(f)$ and
$\coker(f)$ are also square-free. Moreover, if
\begin{equation*}
\label{eq:MMM}
\begin{tikzcd}[column sep=16pt]
0\ar[r]& M' \ar[r]& M \ar[r]&  M'' \ar[r]& 0
\end{tikzcd}
\end{equation*}
is an exact sequence of $\N^n$-graded $S$-modules, and $M'$
and $M''$ are  square-free, then so is $M$.
\end{proposition}

Proposition \ref{prop:Nn-sqf-ker-coker} has a few interesting consequences.

\begin{corollary}
\label{cor:sqf-res}
Let $M$ be an $\N^n$-graded square-free $S$-module.
Then all the modules in the  minimal free $\N^n$-graded
resolution of $M$ are square-free.
\end{corollary}

\begin{corollary}
\label{cor:Nn-sqf-homology}
If $\mathbf{F}$ is a bounded complex of free square-free $S$-modules,
then the homology modules of $\mathbf{F}$ are also square-free.
\end{corollary}

The following result will be of particular interest for us.

\begin{theorem}
\label{thm:annihilator}
If $M$ is an $\N^n$-graded, square-free $S$-module,
then its annihilator is a square-free monomial ideal.
In particular, the annihilator of $M$ is a radical ideal.
\end{theorem}

\begin{proof}
Since $M$ is an $\N^n$-graded $S$-module,
the annihilator $\Ann(M) \subseteq S$ is also
$\N^n$-graded, that is, it is a monomial ideal.
Let $m=x_{1}^{a_1}\cdots x_{n}^{a_n}\in\Ann(M)$
be a monomial annihilating $M$, and assume
$a_k>1$ for some $k$.
Then the multiplication map
\[
\begin{tikzcd}[column sep=16pt]
	m \colon M_{\mathbf{b}}\ar[r]&  M_{{\mathbf{b}}+\deg(m)}
\end{tikzcd}
\]
is zero for all $\mathbf{b} \in \N^n$. We have
$k \in \Supp(\mathbf{b}+\deg(m)-\mathbf{e}_k)$, and
so, by hypothesis, the  map
\[
	\begin{tikzcd}[column sep=16pt]
		x_k  \colon M_{\mathbf{b}+\deg(m)-\mathbf{e}_k}\ar[r]&
		M_{\mathbf{b}+\deg(m)}
	\end{tikzcd}
\]
is an isomorphism. Therefore,
\[
	\begin{tikzcd}[column sep=16pt]
		m/{x_k} \colon M_{\mathbf{b}}\ar[r]& M_{\mathbf{b}+
			\deg(m)-\mathbf{e}_k}
	\end{tikzcd}
\]
is the zero map for all $\mathbf{b} \in \N^n$, and
thus $x_{1}^{a_1}\cdots x_k^{a_k-1}\cdots x_{n}^{a_n}\in\Ann(M)$.
By repeating the argument, we see that $\Ann(M)$ is a square-free
monomial ideal.
\end{proof}

Finally, we note that Theorem \ref{thm:annihilator} and \cite[Lemma 2.2]{Ya00} give the following.

\begin{proposition}
\label{prop:sfq-support}
Let $M$ be a finitely-generated $\N^n$-graded, square-free $S$-module. Then the annihilator
scheme structure on the support of $M$ is reduced. Moreover, the decomposition of the
support in irreducible components is given by
	\[
	\Supp(M)=\bigcup_{\substack{
		\textup{$\mathbf{b}$ square-free}\\ \textup{maximal with $M_\mathbf{b}\ne 0$}
	}}\k ^{\Supp(\mathbf{b})},
	\]
where $\k^\mathsf{V'}$ denotes the locus $V(x_i|\ i\not\in \mathsf{V}')$.
\end{proposition}

Proposition \ref{prop:sfq-support} will be essential for describing the components
of the support resonance loci of a simplicial complex in the next section.

We end this section with the following definition:

\begin{definition}
\label{defn:sqf-part}
For an $\N^n$-graded vector space $M$, the \emph{square-free part}\/ of $M$ is the
subspace $\sqf (M)\subseteq M$ concentrated in square-free multidegrees.
\end{definition}

\subsection{Stanley--Reisner rings}
\label{subsec:SR-ring}
Let $S=\k[x_1,\dots, x_n]$ be the polynomial ring in $n$
variables over a field $\k$ of characteristic $0$.
Given a simplicial complex $\Delta$ on $n$ vertices,
we let $\k[\Delta] \coloneqq S/I_{\Delta}$ be the (polynomial)
{\em Stanley--Reisner ring}\/ of $\Delta$, where $I_{\Delta}$ is the ideal
generated by the (square-free) monomials $x_{\sigma}=x_{i_1} \cdots x_{i_s}$
for all simplices $\sigma=(i_1,\dots, i_s)$ with $1\le i_1<\cdots <i_s \le n$
not in $\Delta$.
Similarly, we define the {\em exterior Stanley--Reisner ring}\/ of $\Delta$ as
$\SR \coloneqq E/J_{\Delta}$, where $E=\bwedge(e_1,\dots, e_n)$ is the exterior algebra
in $n$ variables over $\k$ and $J_{\Delta}$ is the ideal generated
by the monomials $e_{\sigma}=e_{i_1}\wedge \cdots \wedge e_{i_s}$
for all simplices $\sigma\notin \Delta$.

Consider the graded, graded-commutative $\k$-algebra $A^{\bullet}\coloneqq \SR$. 
As mentioned in Remark \ref{rem:grading}, this algebra is given the positive grading.
In each degree $d$, the vector space $A^d$ is spanned
by multivectors $e_\sigma$, where $\sigma$ is a $(d-1)$-dimensional
face of $\Delta$. Indeed, since $\sigma=(i_1,\ldots,i_s)\not\in\Delta$
implies $(i_1,\ldots,i_s,j)\not\in\Delta$ for all $j\not\in\Supp(\Delta)$,
it follows that in each degree $d$, the vector space $J_{\Delta,d}$
is spanned by the multivectors $e_\sigma$ with $\sigma\not\in\Delta$
of dimension $d-1$. With the notation of the previous sections, the
dual $A_d$ is generated by the vectors $v_\sigma$ with $\sigma\in\Delta$
being of dimension $d-1$.

For an element $a=\sum_{i=1}^{n} \lambda_i e_i\in A^1$,
let $(A^{\bullet}, \delta_a)$ be the cochain complex from \eqref{eq:cc}.
As shown in \cite[Proposition 4.3]{AAH} (see also \cite[Lemma 3.4]{PS-adv09}),
this complex depends only on $\Supp(a)\coloneqq \{i\mid \lambda_i\ne 0\}$;
more precisely, $(A^{\bullet}, \delta_a)$ is isomorphic to $(A^{\bullet}, \delta_{\bar{a}})$,
where $\bar{a}=\sum_{i\in \Supp(a)} e_i$. The following Hochster-type
formula from \cite[Proposition 4.3]{AAH}, suitably interpreted and corrected
in \cite[Proposition 3.6]{PS-adv09}, describes the cohomology groups
of the cochain complexes $(A^{\bullet}, \delta_a)$.

\begin{proposition}[\cites{AAH, PS-adv09}]
\label{prop:aah-hochster}
Let $\Delta$ be a finite simplicial complex on vertex set $\mathsf{V}=[n]$
and $a\in A^1$ as above. Writing $\mathsf{V}'=\Supp(a)$, we have
\[
\dim_{\k} H^{i}\bigl(\SR,\delta_a\bigr)=\sum_{\sigma\in \Delta_{\mathsf{V}\setminus \mathsf{V}'}}
\dim_{\k} \widetilde{H}_{i-1-\abs{\sigma}} \bigl(\lk_{\Delta_{\mathsf{V}'}}(\sigma); \k\bigr).
\]
\end{proposition}

Here $\Delta_{\mathsf{V}'}\coloneqq \{\tau\in \Delta\mid \tau \subset \mathsf{V}'\}$ is the
simplicial complex obtained by restricting $\Delta$ to $\mathsf{V}'$ and
$\lk_{\Delta_{\mathsf{V}'}}(\sigma)\coloneqq
\{\tau \in \Delta_{\mathsf{V}'} \mid \tau\cup \sigma \in \Delta\}$
is the link of a simplex $\sigma$ in $\Delta_{V'}$. The range of summation
in the above formula includes the empty simplex,
with the convention that $\abs{\emptyset}=0$ and
$\widetilde{H}_{-1}(\emptyset; \k)=\k$.

\subsection{Koszul modules of a simplicial complex}
\label{subsec:K-simp}

Fix a basis $ v_1, \ldots, v_n$ of the $\k$-vector space $V$.
Let $\mathbf{K}_\bullet$ denote the Koszul complex of $x_1, \ldots, x_n$,
whose $i$-th free module is $\mathbf{K}_i = \bwedge^i V \otimes_{\k} S$,
and set $\deg(v_i) = \mathbf{e}_i \in \N^n$. Then  $\mathbf{K}_\bullet$ is
a complex of $\N^n$-graded square-free $S$-modules.

A simplicial complex $\Delta$ on vertex set $[n]=\{1, \ldots, n\}$ determines a
subcomplex $\mathbf{K}^\Delta_\bullet$ of $\mathbf{K}_\bullet$, whose
$i$-th module $\mathbf{K}_i^\Delta$ is the free $S$-module generated by
the exterior monomials $v_{j_1}\wedge \cdots \wedge v_{j_i}$ such that
$\{j_1, \ldots, j_i\}$ is a face of $\Delta$.
Applying (\ref{eq:wi-a}), the $i$-th \emph{Koszul module}\/ $\Wb_i(\Delta)$ defined
as the $i$-th Koszul module of the exterior Stanley--Reisner ring of $\Delta$ is the
$i$-th homology $H_i(\mathbf{K}^\Delta_\bullet)$.

\begin{proposition}
\label{prop:koszul-sqf}
For every simplicial complex $\Delta$ on $n$ vertices and for every $i$,
the Koszul module $\Wb_i(\Delta)$ is an $\N^n$-graded square-free $S$-module.
\end{proposition}

\begin{proof}
The subcomplex $\mathbf{K}^\Delta_\bullet$ is a complex of
$\N^n$-graded square-free $S$-modules.
By Corollary \ref{cor:Nn-sqf-homology}, it follows that each
homology vector space $\Wb_i(\Delta)$ is a square-free $S$-module.
\end{proof}

\section{Hilbert series for Koszul modules of simplicial complexes}
\label{sect:reg-hilbert}

We fix some notation first. For a multidegree $\mathbf{b}$, we
denote the sum of its entries by  $\abs{\mathbf{b}}$.
For a square-free multidegree $\mathbf{b}\in \N^n$,
we denote by $\Delta_\mathbf{b}$ the restriction of the simplicial complex
$\Delta$ to the subset of the vertices $\Supp(\mathbf{b}) \subseteq \{1, \ldots, n  \}$.

We denote by $\Tilde{h}_i(-; \k)$ and $\Tilde{h}^i(-; \k)$ the dimensions of the simplicial
homology groups $\Tilde{H}_i(-; \k)$ and of the reduced cohomology groups
$\Tilde{H}^i(-; \k)\cong \Tilde{H}_i(-; \k)^\vee$  with
coefficients in $\k$, respectively.

\subsection{Koszul modules vs. Koszul (co)homology}
\label{subseq:K vs K}
We establish a duality result between the Koszul modules associated
to a simplicial complex and Koszul (co)homology of the symmetric 
Stanley--Reisner algebra.

\begin{theorem}
\label{thm:KoszModules-KoszCohom}
For any $i\ge 1$ and  any square-free multi-index $\mathbf{b}$, there are
natural isomorphisms of vector spaces
\begin{equation}
	\label{eqn:Mod-Tor}
	\left[\Wb_i(\Delta)\right]_\mathbf{b}
	\cong \left[\Tor^S_{|\mathbf{b}|-i}
	(\k, \k[\Delta])\right]_\mathbf{b}^\vee\cong \widetilde{H}^{i-1}(\Delta_\mathbf{b}; \k)^\vee\cong
	\widetilde{H}_{i-1}(\Delta_\mathbf{b}; \k).
\end{equation}	
\end{theorem}

\begin{proof}
For the square-free multidegree $\mathbf{b}$, we denote $j\coloneqq \abs{\mathbf{b}} - i$.

We start by proving the first isomorphism.
We use the notation from Section \ref{subsec:SR-ring}. Let $A_d\subseteq \bwedge^d V$
be the subspace generated by the exterior monomials $v_\sigma$ such that
$\sigma$ is a face of $\Delta$.
Denote by $S_d$ the graded component of $S=\Sym(V)$
of total degree $d$. Then, the vector space $[\Wb_i(\Delta)]_\mathbf{b}$
is the middle homology of the complex of vector spaces
	\begin{equation*}
		\label{eq:seq1}
		\begin{tikzcd}[column sep=16pt]
			{[A_{i+1}\otimes S_{j-1}]}_\mathbf{b}
			\ar[r] &
			{[A_{i}\otimes S_{j}]}_\mathbf{b}
			\ar[r] &
			{[A_{i-1}\otimes S_{j+1}]}_\mathbf{b}.
		\end{tikzcd}
	\end{equation*}
By Proposition \ref{prop:koszul-sqf}, this complex is the same as
	\begin{equation*}
		\label{eq:seq2}
		\begin{tikzcd}[column sep=16pt]
			{[A_{i+1}\otimes \sqf(S_{j-1})]}_\mathbf{b}
			\ar[r]&
			{[A_{i}\otimes \sqf(S_{j})]}_\mathbf{b}
			\ar[r]&
			{[A_{i-1}\otimes \sqf(S_{j+1})]}_\mathbf{b}.
		\end{tikzcd}
	\end{equation*}
Upon identifying  $\sqf(S_{d})= \bigwedge^d V^\vee$ via the $\k$-linear map given by 
$x_{i_1}\cdots x_{i_d}\mapsto e_{i_1}\wedge\cdots\wedge e_{i_d}$, 
this chain complex may be written as
	\begin{equation}
		\begin{tikzcd}[column sep=18pt]
			\left[A_{i+1}\otimes \bigwedge^{j-1} V^\vee\right]_\mathbf{b}
			\ar[r]&
			\left[A_{i}\otimes \bigwedge^{j} V^\vee\right]_\mathbf{b}
			\ar[r]&
			\left[A_{i-1}\otimes \bigwedge^{j+1} V^\vee\right]_\mathbf{b},
		\end{tikzcd}
	\end{equation}
	which, by dualization gives
		\begin{equation}
			\label{eq:seq3}
		\begin{tikzcd}[column sep=18pt]
			\left[A^{i-1}\otimes \bigwedge^{j+1} V\right]_\mathbf{b}
			\ar[r]&
			\left[A^{i}\otimes \bigwedge^{j} V\right]_\mathbf{b}
			\ar[r]&
			\left[A^{i+1}\otimes \bigwedge^{j-1} V\right]_\mathbf{b}.
		\end{tikzcd}
	\end{equation}

As usual, all these identifications are made without changing the positivity of the grading.
After having identified $A^d = \sqf(S/I_\Delta)_d$, the sequence \eqref{eq:seq3}
may be written as
	\begin{equation*}
		\label{eq:seq4}
		\begin{tikzcd}[column sep=16pt]
			\left[\sqf(S/I_\Delta)_{i-1}\otimes \bigwedge^{j+1} V\right]_\mathbf{b}
			\ar[r]&
			\left[\sqf(S/I_\Delta)_{i}\otimes \bigwedge^{j} V\right]_\mathbf{b}
			\ar[r]&
			\left[\sqf(S/I_\Delta)_{i+1}\otimes \bigwedge^{j-1} V\right]_\mathbf{b}.
		\end{tikzcd}
	\end{equation*}
	Since $\mathbf{b}$ is a square-free multidegree, this complex is the same as
	\begin{equation}
		\label{eq:seq5}
		\begin{tikzcd}[column sep=16pt]
			\left[(S/I_\Delta)_{i-1}\otimes \bigwedge^{j+1} V\right]_\mathbf{b}
			\ar[r]&
			\left[(S/I_\Delta)_{i}\otimes \bigwedge^{j} V\right]_\mathbf{b}
			\ar[r]&
			\left[(S/I_\Delta)_{i+1}\otimes \bigwedge^{j-1} V\right]_\mathbf{b}.
		\end{tikzcd}
	\end{equation}
	
By the properties of the Koszul complex, the middle cohomology of this complex
is isomorphic to
\[
\left[\Tor^S_{j} (\k, \k[\Delta])\right]_\mathbf{b},
\]
and this concludes the proof of the first isomorphism.

For the second isomorphism, note that
$\left[\Tor^S_j(\k, \k[\Delta])\right]_\mathbf{b}$ is isomorphic to
$\left[\Tor^S_{j-1}(\k, I_\Delta)\right]_\mathbf{b}$. Indeed, for $j\ge 2$,
this is clear, whereas for $j=1$ this follows from the fact the $I_\Delta$ is contained
in the ideal generated by the variables, and hence $\k\cong\k\otimes_S \k[\Delta]$.
From \cite[Proof of Theorem 8.1.1]{HerzogHibi} we obtain an isomorphism
$\left[\Tor^S_{j-1}(\k, I_\Delta)\right]_\mathbf{b}\cong \tilde{H}^{|\mathbf{b}|-j-1}(\Delta_\mathbf{b}; \k)$.
In conclusion,
\begin{equation}
	[\Wb_i(\Delta)]_\mathbf{b}\cong \tilde{H}^{i-1}(\Delta_\mathbf{b}; \k)^\vee,
\end{equation}
as soon as $|\mathbf{b}|-i\ge 1$.
\end{proof}

\begin{remark}
\label{rem:HH}
The isomorphism in the statement of Theorem \ref{thm:KoszModules-KoszCohom}
does not necessarily hold if we drop the
hypothesis that $\mathbf{b}$ is square-free. Indeed, $\left[\Tor^S_{|\mathbf{b}|-i}
(\k, \k[\Delta])\right]_\mathbf{b}$ is equal to $0$ if $\mathbf{b}$ is not square-free,
\cite[Theorem 8.1]{HerzogHibi}.
On the other hand, since the square-free multi-indices are finitely many, the vanishing of
$\left[\Wb_i(\Delta)\right]_\mathbf{b}$ for all $\mathbf{b}$ that is not square-free implies
$\Wb_i(\Delta)$ is of finite length.
\end{remark}

An alternate, less explicit proof of the above theorem can be obtained
by applying the Bernstein--Gelfand--Gelfand
correspondence to express $[\Wb_i(\Delta)]_\mathbf{b}$ as the (duals) of some $\Tor$ spaces
over the exterior algebra, and then apply a theorem of Aramova, Avramov, and Herzog \cite{AAH},
see \cite[Corollary 7.5.2]{HerzogHibi}. More precisely, we have the following result.

\begin{proposition}
\label{prop:BGG}
For any $i\ge 1$ and any square-free multi-index $\mathbf{b}$, there is a natural isomorphism
of vector spaces
	\begin{equation}
		\label{eqn:Mod-Tor-E}
		\left[\Wb_i(\Delta)\right]_\mathbf{b}
		\cong \left[\Tor^E_{|\mathbf{b}|-i}
		\bigl(\k, \SR\bigr)\right]_\mathbf{b}^\vee.
	\end{equation}	
\end{proposition}

The proof of the proposition follows from an adaptation to the multi-graded
context \cite{Brown-Erman} of the classical BGG correspondence, as
described in \cite{Ei-syz}.

\subsection{Multigraded Hilbert series}
\label{subsec:hilb}

Our next goal is to determine the Hilbert series of the Koszul modules $\Wb_i(\Delta)$ 
associated to a simplicial complex $\Delta$. For a multidegree $\mathbf{a}=(a_1, \ldots, a_n)\in \N^n $,
we will write $\mathbf{t}^\mathbf{a} \coloneqq t_1^{a_1 }\cdots t_n^{a_n}$.

\begin{theorem}
\label{thm:hilb-wd}
For every simplicial complex $\Delta$ and every $i>0$,
the  $\N^n$-graded Hilbert series of the Koszul module 
$\Wb_i(\Delta)$ is given by
\begin{equation*}
\label{eq:hilb-koszul}
 \sum_{\mathbf{a}\in \N^n} \dim_{\k} [\Wb_i(\Delta)]_\mathbf{a} \,
\mathbf{t}^\mathbf{a}=
 \sum_{\substack{\mathbf{b}\,\in\, \N^n\\ \mathbf{b} \operatorname{square-free}}}
\dim_{\k}(\widetilde{H}_{i-1}(\Delta_\mathbf{b}; \k))
\frac{\mathbf{t}^\mathbf{b}}{\prod_{j \in \Supp(\mathbf{b})}(1-t_j)}.
\end{equation*}
\end{theorem}

\begin{proof}
We begin by observing that, by the definition of a square-free module,
	we have
	\begin{equation}
		\label{eq:hilb-multi}
		\sum_{\mathbf{a}\in \N^n} \dim [\Wb_i(\Delta)]_\mathbf{a} \, \mathbf{t}^\mathbf{a}
		=
		\sum_{\substack{\mathbf{b}\in \N^n\\ \mathbf{b} \text{ square-free}}}
		\dim [\Wb_i(\Delta)]_\mathbf{b} \,
		\frac{\mathbf{t}^\mathbf{b}}{\prod_{j \in \Supp(\mathbf{b})}(1-t_j)} \, .
	\end{equation}
Thus, it suffices to determine $\dim [\Wb_i(\Delta)]_\mathbf{b}$
when $\mathbf{b}$ is a square-free multidegree.
Fix a square-free multidegree $\mathbf{b}$,
and let $j = \abs{\mathbf{b}} - i$.
From Theorem \ref{thm:KoszModules-KoszCohom}, we know that
\begin{equation}
\label{eq:wbi-Delta}
	[\Wb_i(\Delta)]_\mathbf{b}\cong \tilde{H}_{i-1}(\Delta_\mathbf{b}; \k).
\end{equation}
Therefore, 
\begin{equation}
\label{eq:hilb-multi-bis}
		\sum_{\mathbf{a}\in \N^n} \dim [\Wb_i(\Delta)]_\mathbf{a} \,
		\mathbf{t}^\mathbf{a}
		=
		\sum_{\substack{\mathbf{b}\in \N^n\\ \mathbf{b} \text{ square-free}}}
		\tilde{h}_{i-1}\bigl(\Delta_\mathbf{b}, \k\bigr)
		\frac{\mathbf{t}^\mathbf{b}}{\prod_{j \in \Supp(\mathbf{b})}(1-t_j)}, 
\end{equation}
and this completes the proof.
\end{proof}

Specializing to the single $\N$-grading, the above theorem yields the following formula 
for the Hilbert series of the Koszul module $\Wb_i(\Delta)$:
\begin{equation}
\label{eq:hilb-single}
\sum_{a\in \N} \dim [\Wb_i(\Delta)]_a \, t^a=
\sum_{\substack{\mathbf{b}\,\in\, \N^n\\ \mathbf{b}\: \operatorname{square-free}}}
\dim \bigl(\widetilde{H}_{i-1}(\Delta_\mathbf{b}; \k)\bigr) \left(\frac{t}{1-t}\right)^{\abs{\mathbf{b}}}.
\end{equation}

In the particular case when $\Delta$ has dimension at most $1$, that is, when $\Delta$ is equal
to a (simplicial) graph $\Gamma$, we recover the Hilbert series of the module
$W_\Gamma\coloneqq \Wb_1(\Gamma)(2)$, as computed in \cite[Theorem 4.1]{PS-mathann}.

\begin{corollary}[\cite{PS-mathann}]
\label{thm:chen-raag}
For a graph $\Gamma$ on vertex set $\mathsf{V}$, we have
\begin{equation*}
\label{eq:hilb-graph}
	\Hilb(W_{\Gamma},t )= \frac{1}{t^2}\cdot Q_{\Gamma}\Big( \frac{t}{1-t} \Big),
\end{equation*}
where $Q_{\Gamma}(t)=\sum_{j\ge 2} c_j(\Gamma) t^j$ and
$c_j(\G)=\sum_{\mathsf{V}'\subseteq \mathsf{V}\colon  \abs{\mathsf{V}'}=j }
\tilde{h}_0(\Gamma_{\mathsf{V}'})$.
\end{corollary}

The significance of the above formula is that it gives the Chen ranks of the
right-angled Artin group $G_{\Gamma}$ associated to the graph $\Gamma$.

\section{Resonance varieties of a simplicial complex}
\label{sec:simp-comp}

Given an (abstract) simplicial complex $\Delta$ on vertex set $\mathsf{V}$,
we define its resonance varieties as those of the corresponding
exterior Stanley--Reisner ring. That is, we put $\RR^i(\Delta)\coloneqq \RR^i(\SR)$
for the jump resonance and $\Rt_i(\Delta)\coloneqq \Rt_i(\SR)$ for the
support resonance varieties, respectively.

Using Proposition \ref{prop:aah-hochster},
a precise description of the varieties $\RR^i(\Delta)$ was given
in \cite[Theorem 3.8]{PS-adv09}, as follows.

\begin{proposition}
\label{prop:irred-comps-jump}
For each $i\ge 1$, the decomposition in irreducible components of the jump
resonance variety is given by
\begin{equation}
	\label{eq:res-delta}
	\RR^i(\Delta)= \bigcup_{\substack{
	          \textup{$\mathsf{V}' \subseteq \mathsf{V}$ maximal such that}\\[1pt]
		\exists \sigma\in \Delta_{\mathsf{V}\setminus \mathsf{V}'}, \
			\widetilde{H}_{i-1-\abs{\sigma}} (\lk_{\Delta_{\mathsf{V}'}}(\sigma);\k) \ne 0
			}}
	 \k^{\mathsf{V}'}.
\end{equation}
\end{proposition}

Here $\k^{\mathsf{V}'}$ denotes the coordinate subspace of $\k^{\mathsf{V}}=\k^n$ (where
$n=\abs{\mathsf{V}}$) spanned by the vectors $\{\mathbf{e}_i  \mid i\in \mathsf{V}'\}$.
On the other hand, for the support resonance defined in (\ref{eq:ria-spec}), the situation
is different in degrees $i>1$.

\begin{theorem}
\label{prop:irred-comps-support}
For each $i\ge 1$, the scheme structure on the support resonance locus $\Rt_i(\Delta)$ is reduced.
Moreover, the decomposition in irreducible components is given by
	\begin{equation}
		\label{eq:res-delta2}
		\Rt_i(\Delta)=\bigcup_{\substack{
			\textup{$\mathsf{V'}\subseteq \mathsf{V}$ maximal with}\\[1pt]
			\widetilde{H}^{i-1}(\Delta_{\mathsf{V'}};\k)\neq 0
		}}
		\k ^{\mathsf{V'}}.
	\end{equation}
\end{theorem}

\begin{proof}
The first claim follows from Proposition \ref{prop:koszul-sqf} and
Theorem \ref{thm:annihilator}. The precise structure of the decomposition in
irreducible components is governed by the multi-graded structure detailed in
Theorem \ref{thm:hilb-wd} and Proposition \ref{prop:sfq-support}.
Observe that (\ref{eq:res-delta2}) corresponds to the primary decomposition
of the ideal $\Ann(\Wb_i(\Delta))$.
\end{proof}

Notice the difference at the set level between (\refeq{eq:res-delta}) and (\refeq{eq:res-delta2});
in particular, observe that the support resonance loci are easier to describe. Furthermore, whereas
Theorem \ref{prop:irred-comps-support} guarantees that the support resonance schemes
$\Rt_i(\Delta)$ are always reduced,
the corresponding jump resonance loci $\RR^i(\Delta)$ are not necessarily reduced
(with the Fitting scheme structure), even in weight one, as the following example illustrates.

\begin{example}
\label{ex:path4}
Let $\Gamma$ be a path on $4$ vertices. Then
$\Fitt_0(W_1(\Gamma))=(x_2)\cap (x_3) \cap (x_1,x_2^2,x_3^2,x_4)$ is
not reduced, although $\Ann(W_1(\Gamma))=(x_2)\cap (x_3)$
is reduced. Therefore, the Fitting scheme structure on $\RR^1(\Gamma)$
has an embedded component at $0$.
\end{example}

A simplicial complex $\Delta$ of dimension $d$ is said to be a
{\em Cohen--Macaulay complex}\/  over $\k$ if
$\widetilde{H}^{\bullet}(\lk(\sigma);\k)$ is concentrated in degree
$d -\abs{\sigma}$, for all $\sigma\in \Delta$. As shown in \cite{DSY},
the jump resonance varieties of such a simplicial complex {\em propagate};
that is,
\begin{equation}
	\label{eq:propagate}
	\RR^1(\Delta) \subseteq \RR^2(\Delta)  \subseteq \cdots  \subseteq
	\RR^{d+1}(\Delta).
\end{equation}
For arbitrary simplicial complexes, though, the resonance varieties
do not always  propagate. This phenomenon, first identified in \cite{PS-adv09},
happens even for graphs.

\begin{example}[\cite{PS-adv09}]
	\label{ex:not-prop}
	Let $\Delta$ be the disjoint union of two edges. Then $\RR^1(\Delta)=\k^4$,
	whereas $\RR^2(\Delta)=\k^2 \cup \k^2$, the union of two transversal
	coordinate planes. Thus, $\RR^1(\Delta) \not\subseteq \RR^2(\Delta)$.
\end{example}

When $\Delta$ is Cohen--Macaulay,
propagation and formula \eqref{eq:union-res} give $\RR^i(\Delta)=\bigcup_{j\le i}
\Rt_j(\Delta)$. But it is not known whether the support resonance varieties  $\Rt_i(\Delta)$ propagate
when $\Delta$ is Cohen--Macaulay, or, equivalently, whether $\RR^i(\Delta)=\Rt_i(\Delta)$
in this case. In general, though, we can use the previous example to settle the
latter question in the negative.

\begin{example}
\label{ex:compare-res}
Let $\Delta$ be the disjoint union of two edges. Then $\Rt_1(\Delta)=\RR^1(\Delta)=\k^4$
but  $\Rt_2(\Delta)=\emptyset$ whereas, as we saw before, $\RR^2(\Delta)=\k^2 \cup \k^2$.
Thus, $\Rt_2(\Delta)\ne \RR^2(\Delta)$.
\end{example}

\section{Regularity and projective dimension for Koszul modules of simplicial complexes}
\label{section:CMreg}

\subsection{General bounds}
\label{subsec:reg}
We start this section with an upper bound on the Castelnuovo--Mumford
regularity and projective dimension of the Koszul modules.

\begin{proposition}
	\label{prop:reg-pdim}
	For every simplicial complex $\Delta$ on $n$ vertices
	and every $i>0$, the Koszul module $\Wb_i(\Delta)$
	has regularity at most $n$ and projective dimension
	at most $n-i-1$.
\end{proposition}

\begin{proof}
By definition, the Koszul module $\Wb_i(\Delta)$ is a sub-quotient of
the module $Z_i \subseteq \bwedge^i V \otimes_{\k} S$
of $i$-th cycles in the Koszul complex of $x_1, \ldots, x_n$.
Since $Z_i$ is generated in degree $i+1$, it follows that
the  degree of any of the generators  of $\Wb_i(\Delta)$ is at least $i+1$.
Let $\mathbf{F}_\bullet$ denote the minimal free resolution of  $\Wb_i(\Delta)$.
By Proposition \ref{prop:koszul-sqf} and Corollary \ref{cor:sqf-res},
$\mathbf{F}_\bullet$ is a complex of $\N^n$-graded $S$-modules
generated in square-free multidegrees, hence,
the total degree of the  generators of each $\mathbf{F}_h$ is at most $n$.
The statement on the regularity follows immediately.
Since the least degree of the generators of $\mathbf{F}_{h+1}$ is strictly
larger than the least degree of the generators of $\mathbf{F}_{h}$,
it follows that $\mathbf{F}_{h} = 0 $ for $h > n-i-1$.
\end{proof}

\subsection{Regularity of Koszul modules for simplicial complexes of special type}
\label{sect:fixed-degree}

We fix integers $n\ge 4$ and $1\le d\le n-3$ and
assume $\Delta$ is a simplicial complex of dimension $d$ on $n$ vertices
whose $(d-1)$-skeleton coincides with that of the full simplex, that is,
\begin{equation}
\label{eq:delta-skeleton}
\Delta^{(d-1)}=\big(2^{[n]}\big)^{(d-1)}.
\end{equation}
For instance, if $d=1$, then $\Delta$ is simply a (simplicial) graph on
$n$ vertices. If $d=2$, then $\Delta$ is obtained from the complete
graph on $n$ vertices by filling in some triangles. For this type of
simplicial complexes that generalize graphs, the nature of the Koszul modules
can be made more precise, as follows.

\begin{proposition}
\label{prop:koszul-d-complex}
For a simplicial complex $\Delta$ as above, the following hold.
 \begin{enumerate}[itemsep = 2.5pt, topsep=-0.5pt]
\item \label{pr1} $\Wb_i(\Delta)=0$ for $i\notin \{d,d+1\}$.
\item \label{pr2} $\Wb_d(\Delta)=\coker\big(\partial ^E_{d+2}+j_{d+1}\otimes_{\k} S\big)$.
\item \label{pr3} $\Wb_{d+1}(\Delta)=\ker\big(\partial ^A_{d+1}\big)$, and hence it is
either zero or torsion-free.
 \end{enumerate}	
\end{proposition}

\begin{proof}
Recall that $\Wb_i(\Delta)=H_i (A_{\bullet} \otimes_{\k} S,\partial)$, where 
$A=\SR$. By condition \eqref{eq:delta-skeleton}, the graded algebra $A=E/J_{\Delta}$ 
satisfies the hypothesis from the beginning of Section \ref{subsec:fixed-deg}. Hence, 
Proposition \ref{prop:Koszul-for-fixed-degree} applies, showing that 
$\Wb_i(\Delta)=0$ for $i < d$. Moreover, since $\dim(\Delta)=d$, 
we have that $A_i=0$ for $i> d+1$, and so $\Wb_i(\Delta)$ also 
vanishes in that range, thereby proving Part \eqref{pr1}. 

Part \eqref{pr2} follows at once from part \eqref{wb2} of 
Proposition \ref{prop:Koszul-for-fixed-degree}, 
 while part \eqref{pr3} follows from the fact that $A_{d+2}=0$.
\end{proof}

Using the explicit presentation of $\Wb_d(\Delta)$ from part \eqref{pr2}, we can 
improve the bound on regularity from Proposition \ref{prop:reg-pdim}.

\begin{proposition}
\label{prop:reg-graph}
With notation as above, we have $\reg \Wb_d(\Delta) \leq n-2$.
\end{proposition}

\begin{proof}
We have a presentation
\begin{equation}
\label{eq:pres}
\begin{tikzcd}[column sep=16pt]
0 \ar[r]& D  \ar[r]& Z_d  \ar[r]& \Wb_d(\Delta)  \ar[r]& 0 ,
\end{tikzcd}
\end{equation}
where $Z_d\subseteq \bwedge^dV \otimes_{\k} S $ is the module of
Koszul $d$-cycles, and $D$ is the image of $K \otimes_{\k} S$ under
the Koszul differential. Both $Z_d$ and $D$ are generated in degree $d+1$.
The module $Z_d$ has a linear free resolution, consisting
of the truncated Koszul complex, so $\reg Z_d =d+1$.
Since the module $D$ is square-free, its syzygy modules
are also square-free, and hence, they are generated in degrees at most $n$.
This implies that $\reg D \leq n-1$, since $D$ does not have generators
of degree $n$. Applying the long exact sequence of $\Tor (-,\k)$ to
\eqref{eq:pres}, we obtain
\begin{equation}
\label{eq:reg-wb}
\reg \Wb_d(\Delta) \leq \max(\reg Z_d, \reg D-1)= \max(d+1,n-2) = n-2.
\end{equation}
and this completes the proof.
\end{proof}

If $d\le 1$, that is, if $\Delta=\Gamma$ is a graph on $n$ vertices, taking into
account the degree shift, we obtain the bound
\begin{equation}
\label{eq:reg-graph}
\reg W_1(\Gamma) \le n-4.
\end{equation}

\begin{example}
\label{ex:n-cycle}
If $\Gamma=\mathcal{C}_n$ is the cycle on $n\ge 4$ vertices,
then the regularity of $W_1(\Gamma)$ attains the above bound:
\[
\reg W_1(\Gamma)=n-4
\quad \text{and} \quad \pdim W_1(\Gamma)= n-2.
\]
This follows from \eqref{eq:pres}, since in this case the module $D$
has only one syzygy, of degree $n$.
\end{example}

\begin{remark}
\label{rem:koszul-wd-delta}
For the Koszul module $\Wb_d(\Delta)$, the simplified presentation
\eqref{eqn:simple-presentation} has the following nice interpretation.
Let $\widetilde{\Delta}$ be the simplicial complex which is maximal 
(with respect to inclusion) among all simplicial complexes that share 
the same $d$-skeleton with $\Delta$; for instance, if $d=1$, then 
$\widetilde{\Delta}$ is the flag complex of the graph $\Gamma=\Delta$. 
Denote by $V_i$ the $\k$-vector space with basis the set of $i$-dimensional 
missing faces of $\widetilde{\Delta}$. We then have an exact sequence,
\begin{equation}
\label{eq:spn-f}
\begin{tikzcd}[column sep=18pt]
V_{d+2}\otimes_{\k} S\ar[r]& V_{d+1}\otimes_{\k} S\ar[r]&
\Wb_d(\Delta)\ar[r]& 0.
\end{tikzcd}
\end{equation}
\end{remark}

Using Proposition \ref{prop:koszul-d-complex}, together with formula \eqref{eq:res-delta}, we
obtain the following immediate corollary.

\begin{corollary}
\label{prop:res-d-complex}
With $\Delta$ as above, we have:
	\begin{enumerate}[itemsep = 2.5pt,topsep=0pt]
		\item $\Rt_i(\Delta)=\RR^i(\Delta)$ for all $i\ne d+1$.
		\item $\Rt_{d+1}(\Delta)$ is equal to either $\emptyset$ or $\k^n$.
		\item $\RR^d(\Delta)= \bigcup_{\substack{\mathsf{V'}\subseteq \mathsf{V}\,
		        \mathrm{maximal}\\[0.5pt]
			{\widetilde{H}_{d-1}(\Delta_{\mathsf{V'}};\k)\ne 0}}} \k^\mathsf{V'}$.
	\end{enumerate}
\end{corollary}

\begin{example}
\label{ex:tetra}
Let $\Delta$ be the boundary of the tetrahedron, with the face $\sigma=\{1,2,3\}$
missing. Then $\Delta^{(1)}=\big(2^{[4]}\big)^{(1)}$, and so $\Delta$
is a simplicial complex covered by the above corollary, with $d=2$.
In this case, we have that $\Rt_d(\Delta)=\{x_4=0\}$, since $H_1(\Delta_{\sigma};\k)=\k$,
and  $\RR^d(\Delta)=\{x_4=0\}$, since
$\widetilde{H}_{2-1-1}(\lk_{\Delta_{\sigma}}(\{4\});\k)=\widetilde{H}_{0}(\emptyset;\k)=\k$.
\end{example}

As already mentioned before, the loci $\Rt_{d+1}(\Delta)$ and $\RR^{d+1}(\Delta)$
can be different, in general. For example, if we take the graph $\Gamma$ on four
vertices with edges $(1,2)$ and $(3,4)$ as in Example \ref{ex:not-prop}, then
$\Rt_2(\Gamma)=\emptyset$ while $\RR^2(\Gamma)=V(x_1,x_2)\cup V(x_3,x_4)$.

\newcommand{\arxiv}[1]
{\texttt{\href{http://arxiv.org/abs/#1}{arXiv:#1}}}
\newcommand{\arx}[1]
{\texttt{\href{http://arxiv.org/abs/#1}{arxiv:}}
\texttt{\href{http://arxiv.org/abs/#1}{#1}}}
\newcommand{\arxx}[2]
{\texttt{\href{https://arxiv.org/abs/#1.#2}{arxiv:#1.}}
\texttt{\href{https://arxiv.org/abs/#1.#2}{#2}}}
\newcommand{\doi}[1]
{\texttt{\href{http://dx.doi.org/#1}{doi:#1}}}
\renewcommand{\MR}[1]
{\href{http://www.ams.org/mathscinet-getitem?mr=#1}{MR#1}}

\bibliographystyle{amsplain}

\end{document}